\documentclass[12pt]{amsart}
\usepackage{setspace, amsmath, amsthm, amssymb, amsfonts, amscd, epic, graphicx, ulem, dsfont}
\usepackage[T1]{fontenc}
\usepackage{multirow}
\usepackage{bbm}
\usepackage{enumerate}

\makeatletter \@namedef{subjclassname@2010}{
  \textup{2020} Mathematics Subject Classification}
\makeatother

\newtheorem{thm}{Theorem}[section]

\newtheorem{pro}[thm]{Proposition}

\theoremstyle{remark}
\newtheorem*{rema}{Remark}

\theoremstyle{definition}

\newcommand{\Ima}{\operatorname{Im}}
\newcommand{\Rea}{\operatorname{Re}}

\newcommand{\R}{\mathbb{R}}

\newcommand{\N}{\mathbb{N}}

\newcommand{\C}{\mathbb{C}}

\begin{document}

\title[Square root theorem for positive matrices]{The square root theorem for positive semidefinite matrices and their monotonicity}
\author[M. A. Aouichaoui and M. H. Mortad]{Mohamed Amine Aouichaoui and Mohammed Hichem Mortad$^*$}

\thanks{* Corresponding author.}
\date{}
\keywords{}

\subjclass[2010]{Primary 15B48. Secondary 15A39, 15A27, 15A18.}

\address{(The first author) University of Monastir, Institute of Preparatory Studies in
Engineering of Monastir, Tunisia, 5019 Monastir, Tunisia.}

\email{amineaouichaoui@yahoo.com, amine.aouichaoui@fsm.rnu.tn}

\address{(The corresponding and second author) Department of
Mathematics, University of Oran 1, Ahmed Ben Bella, B.P. 1524, El
Menouar, Oran 31000, Algeria.}

\email{mhmortad@gmail.com, mortad.hichem@univ-oran1.dz.}

\begin{abstract}In this note, simple proofs of certain
well-known results involving the positive square root of positive
matrices are provided.
\end{abstract}

\maketitle

\section{Introduction}

Let $n\in\N$. A matrix $A$, defined on $\C^n$, is called positive
semidefinite when $\langle Ax,x\rangle\geq 0$ for all $x\in\C^n$.
This implies that such matrix must be self-adjoint, that is,
$A=A^*$. Also, self-adjoint matrices always possess eigenvalues.
Moreover, if $A$ is a self-adjoint matrix, then $A=0$ if and only if
$A$ has one unique eigenvalue, namely $\lambda=0$.

If $T$ is a matrix, which is defined on $\C^n$, then we may always
express it as $A+iB$ where $A$ and $B$ are two self-adjoint
matrices. In such case, $A=(T+T^*)/2:=\Rea T$ and
$B=(T-T^*)/{2i}:=\Ima T$, and they are called the real and imaginary
parts of $T$ respectively. Recall that a square root of a matrix $A$
is any matrix $B$ that obeys $B^2=A$. In the end, if $A$ is positive
demidefinite, then $\langle Ax,x\rangle=0$ gives $Ax=0$.

When teaching an advanced linear algebra course (or a similar
course), the following theorem is usually shown.

\begin{thm}\label{sqroot THM}
Let $A$ be a positive semidefinite matrix. Then
\begin{enumerate}
  \item $A$ possesses a unique positive semidefinite square root, noted
$\sqrt A$.
  \item If $B$ is a matrix such that $BA=AB$, then $B\sqrt
A=\sqrt A B$.
\end{enumerate}
\end{thm}

While the proof of existence is straightforward, uniqueness,
however, is usually not very straightforward. Here, we include a
proof that first appeared in \cite{Koeber-Schafer}, which, and as
far as we are aware, is the most elegant proof that exists in the
literature. The last claim of Theorem \ref{sqroot THM} is also
well-known, but again we present a simple proof, which is based on
simultaneous diagonalization of self-adjoint matrices (see, e.g.,
Exercise 14 in Chapter in \cite{Friedberg-Insel-Spence-LIN-ALG} for
the latter concept).

Then we include a simple proof of the monotonicity of the square
root. Recall in passing that P. R. Halmos made the following comment
in Solution 154 in \cite{Halmos-Linear-Algebra-pb-book-1995}:
Various proofs of "this result" can be constructed, but none of them
jumps to the eye. Therefore, the proof given below deserves to be
shared with the mathematical community. In fine, students and
lecturers alike might be interested in consulting a few references
that deal with similar subjects in both finite and
infinite-dimensional settings. See
\cite{Conway-Morrel-Roots-Logarithms}, \cite{Dehimi-Mortad-Reid},
\cite{Mortad-Oper-TH-BOOK-WSPC}, \cite{Mortad-square-root-normal},
\cite{Mortad-cex-BOOK}, \cite{Mortad-JMAA-p(A)-2023},
\cite{Putnam-square-rt-s-a-logarithm},
\cite{Radjavi-Rosenthal-sq-roots-normal} and
\cite{Shamir-Cohen-Root-matrices-Optics}. For useful references on
general matrix theory, readers are referred to \cite{Axler LINEAR
ALG DONE RIGHT}, \cite{Friedberg-Insel-Spence-LIN-ALG},
\cite{Garcia-Horn-BOOK}, \cite{Halmos-Linear-Algebra-pb-book-1995},
and \cite{Prasolov-Lin-Alg-Book}.

\section{Simple proofs of standard results involving positive semidefinite square roots of matrices}

\begin{proof}[Proof of Theorem \ref{sqroot THM}]\hfill
\begin{enumerate}
  \item As it is known, the existence is obtained using the spectral
theorem. In other words, a square root of $A$ is given by $UDU^*$
for a certain unitary matrix $U$ and where $D$ is a diagonal matrix
whose diagonal has solely the positive square roots of the positive
eigenvalues of $A$.

The proof of uniqueness is borrowed from \cite{Koeber-Schafer} with
a slight modification. Let $R$ and $S$ be two positive semidefinite
matrices obeying $R^2=S^2=A$. W have to show that $R=S$. Since $R-S$
is self-adjoint, to show that $R=S$, it suffices to show that the
only eigenvalue of $R-S$ is $\lambda=0$. So, let $\lambda\in\R$ be
such that $Rx=Sx+\lambda x$ for some nonzero $x$. Then
\[0=\langle(R^2-S^2)x,x\rangle=\langle R(R-S)x,x\rangle+\langle(R-S)Sx,x\rangle=\langle\lambda x,Rx\rangle+\langle Sx,\lambda x\rangle.\]
Hence
\[\lambda(\langle Rx,x\rangle+\langle Sx,x\rangle)=0.\]
So, either $\lambda=0$ or $\langle Rx,x\rangle+\langle
Sx,x\rangle=0$. But, in the latter case, we get $\langle
Rx,x\rangle=\langle Sx,x\rangle=0$ for $R$ and $S$ are both positive
semidefinite and so $Rx=Sx=0$. So, $\lambda x=0$, i.e. $\lambda=0$
again. Thus, we always end up with $\lambda=0$.
  \item If $BA=AB$, then $AB^*=B^*A$ and so
  $A(\Rea B)=(\Rea B)A$ and $A(\Ima B)=(\Ima B)A$. Since $A$ and $\Rea
B$ are self-adjoint, they are simultaneously diagonalizable, i.e.,
for the same unitary matrix $U$, $U^*AU=E$ and $U^*(\Rea B)U=F$
where $E$ and $F$ are two diagonal matrices (with the eigenvalues of
$A$ and $\Rea B$ on their diagonals respectively). Hence
$\sqrt{A}=U\sqrt E U^*$. Since $\sqrt E$ obviously commutes with
$F$, it is seen that $\sqrt A(\Rea B)=(\Rea B)\sqrt A$. An akin
reasoning then allows us to obtain $\sqrt A(\Ima B)=(\Ima B)\sqrt
A$. Thus, $B\sqrt A=\sqrt A B$, as desired.
\end{enumerate}
\end{proof}

The coming result is also well-documented, but the proof here is
once again greatly simple (cf.
\cite{Halmos-Linear-Algebra-pb-book-1995}, Problem 154).

\begin{pro}
Assume $A$ and $B$ are self-adjoint matrices that satisfy $A\geq
B\geq0$. Then $\sqrt A\geq \sqrt B$.
\end{pro}

\begin{proof}The matrix $\sqrt A-\sqrt B$, which is already self-adjoint,
is positive semidefinite when its eigenvalues are all positive. So,
let $\lambda\in\R$ be such that $(\sqrt A-\sqrt B)x=\lambda x$ for
some nonzero $x$. The aim now is to obtain $\lambda\geq0$, and we
provide two simple ways of doing that.

\begin{enumerate}
  \item \textbf{First proof:} We have \[0\leq\langle(A-B)x,x\rangle=\langle\sqrt A(\sqrt A-\sqrt B)x,x\rangle+\langle(\sqrt A-\sqrt B)\sqrt Bx,x\rangle\]
and so $\langle(\sqrt A-\sqrt B)x,\sqrt A x\rangle+\langle\sqrt
Bx,(\sqrt A-\sqrt B)x\rangle\geq 0$ or merely $\lambda(\langle\sqrt
Ax,x\rangle+\langle x,\sqrt Bx\rangle)\geq0$. So,

\begin{enumerate}
  \item If $\langle\sqrt Ax,x\rangle+\langle x,\sqrt Bx\rangle\neq0$, then
$\lambda\geq0$;
  \item and if $\langle\sqrt Ax,x\rangle+\langle x,\sqrt
Bx\rangle=0$, then $\langle\sqrt Ax,x\rangle=\langle x,\sqrt
Bx\rangle=0$ due to the semidefinite positiveness of both $\sqrt A$
and $\sqrt B$. Hence $\sqrt Ax=\sqrt Bx=0$, thereby $\lambda x=0$
from which we derive $\lambda=0$, as desired.
\end{enumerate}
Accordingly, and in either case, $\lambda\geq0$, as needed.
  \item \textbf{Second proof:} We show that the eigenvalues of $\sqrt A-\sqrt B$ are nonnegative.
  Let $(\lambda, x)$ be an eigenpair: $(\sqrt A-\sqrt B) x=\lambda x$. So $\sqrt B x=\sqrt A x-\lambda x, \lambda \in \mathbb{R}$, and
$$
\langle \sqrt A\sqrt B x,x\rangle=\langle \sqrt A (\sqrt A x-\lambda
x),x\rangle=\langle Ax, x\rangle-\lambda \langle \sqrt A x,x\rangle
\in \mathbb{R} .
$$
By the Cauchy-Schwarz inequality, we obtain
$$
\begin{aligned}
\langle Ax,x\rangle & =\langle Ax,x\rangle^{1 / 2}\langle Ax,x\rangle^{1 / 2} \geq \langle Ax,x\rangle^{1 / 2}\langle Bx,x\rangle^{1 / 2} \\
& \geq\left|\langle \sqrt A\sqrt B x,x\rangle\right| \geq \langle \sqrt A\sqrt B x,x\rangle \\
& =\langle Ax, x\rangle-\lambda \langle \sqrt A x,x\rangle.
\end{aligned}
$$
Therefore, $\lambda \langle \sqrt A x,x\rangle \geq 0$. If
$\lambda<0$, we deduce $\langle \sqrt A x,x\rangle=0$ which then
yields $\langle \sqrt B x,x\rangle=0$ for $A\geq B\geq0$. Hence
$\sqrt A x=\sqrt B x=0$, whereby

\[0=\sqrt B x=\sqrt Ax-\lambda x=-\lambda x,\]
which contradicts the fact that $\lambda < 0$ and $x \neq 0$.
\end{enumerate}
\end{proof}

\begin{rema}
Readers may wish to consult \cite{Dehimi-Mortad-Reid} for fairly
simple proofs in infinite-dimensional settings compared to some
others in the existing literature (as, e.g., Exercise 19 on Page 310
of \cite{Costara-Popa-exos-func-analysis-2003-trans-from-Russian} or
Proposition 4.2.8 of \cite{Kadison-Ringrose-I-Cours}).
\end{rema}

\begin{rema}
Recall that if $A$ and $B$ are $n$-square positive semidefinite
matrices, then there exists an invertible matrix $P$ such that $$
P^{*} A P \quad \text { and } \quad P^{*} B P
$$ are both diagonal matrices (see \cite{Newcomb}).

Using this result, by writing $A=P^{*} D_{1} P$ and $B=P^{*} D_{2}
P$, where $P$ is an invertible matrix, and $D_{1}$ and $D_{2}$ are
diagonal matrices with nonnegative entries, we derive other results
of monotonicity for positive semidefinite matrices. That's, if $A
\geq B \geq 0$, then
\begin{enumerate}
\item $\operatorname{rank}(A) \geq \operatorname{rank}(B).$

\item $\operatorname{det}(A) \geq \operatorname{det}(B).$

\item $\operatorname{tr}(A) \geq \operatorname{tr}(B)$, with equality if and only if $A = B.$

\item $B^{-1} \geq A^{-1}$ if $A$ and $B$ are nonsingular.
\end{enumerate}
First, observe that $A \geq B \geq 0$ is equivalent to $D_{1} \geq
D_{2} \geq 0.$

\vspace{2mm}

$(1)$ Since $P$ is an invertible matrix, it suffices to prove
$\operatorname{rank}(D_{1}) \geq \operatorname{rank}(D_{2}).$ Put $
  D_{1} = \operatorname{diag}\left(\lambda_{1}, \lambda_{2}, \ldots, \lambda_{n}\right)$ and $
  D_{2} = \operatorname{diag}\left(\mu_{1}, \mu_{2}, \ldots, \mu_{n}\right).$ Since $D_{1} \geq D_{2} \geq 0,$ it follows that $\lambda_{k} \geq \mu_{k} \geq 0,$ for all $k = 1, 2, \ldots, n.$ And so $\#\{ 1 \leq k \leq n : \, \mu_{k} \neq 0 \} \leq \#\{ 1 \leq k \leq n: \, \lambda_{k} \neq 0 \}.$ This implies $\operatorname{rank}(D_{1}) \geq \operatorname{rank}(D_{2}).$

\vspace{2mm}

$(2)$  $\operatorname{det}(A) = |\operatorname{det}(P)|^{2}
\prod_{k=1}^n \lambda_{k} \geq |\operatorname{det}(P)|^{2}
\prod_{k=1}^n \mu_{k}= \operatorname{det} B.$

\vspace{2mm}

$(3)$ Since the matrix $A-B$ is positive semidefinite, its
eigenvalues are non-negative, and so $\operatorname{tr}(A-B) \geq
0.$ The result follows by linearity of $\operatorname{tr}.$

\vspace{2mm}

$(4)$ If $A$ and $B$ are nonsingular, then we have $\lambda_{k} \geq
\mu_{k} > 0,$ for all $k = 1, 2, \ldots, n.$ This implies that $0 <
\frac{1}{\lambda_{k}} \leq \frac{1}{\mu_{k}},$ for all $k = 1, 2,
\ldots, n.$ Hence, $D_{2}^{-1} \geq D_{1}^{-1},$ which implies that
$B^{-1} \geq A^{-1}.$

\end{rema}

\section*{Conclusion}

The above results, which are very well known, may now be easily
included in our lectures. Besides, the given proofs are much easier
for students to assimilate.


\begin{thebibliography}{1}

\bibitem{Axler LINEAR ALG DONE RIGHT}
S. Axler. \textit{Linear algebra done right}. Third edition.
Undergraduate Texts in Mathematics. Springer, Cham, 2015.

\bibitem{Conway-Morrel-Roots-Logarithms}
J. B. Conway, B. B. Morrel, Roots and logarithms of bounded
operators on Hilbert space, \textit{J. Funct. Anal.}, \textbf{70/1}
(1987) 171-193.

\bibitem{Costara-Popa-exos-func-analysis-2003-trans-from-Russian}
C. Costara, D. Popa. Exercises in functional analysis, Kluwer Texts
in the Mathematical Sciences, \textbf{26}, \textit{Kluwer Academic
Publishers Group, Dordrecht,} 2003.

\bibitem{Dehimi-Mortad-Reid}
S. Dehimi, M. H. Mortad. Generalizations of Reid inequality,
\textit{Mathematica Slovaca}, \textbf{68/6} (2018) 1439-1446.

\bibitem{Friedberg-Insel-Spence-LIN-ALG}
S. H. Friedberg, A. J. Insel, L. E. Spence. \textit{Linear algebra}.
Second edition. Prentice Hall, Inc., Englewood Cliffs, NJ, 1989.

\bibitem{Garcia-Horn-BOOK}
S. R. Garcia, R. A. Horn. \textit{A second course in linear
algebra}, Cambridge Mathematical Textbooks,Cambridge University
Press (2017).

\bibitem{Halmos-Linear-Algebra-pb-book-1995}
P. R. Halmos. \textit{Linear algebra problem book}, The Dolciani
Mathematical Expositions, \textbf{16}. Mathematical Association of
America, Washington, DC, 1995.

\bibitem{Kadison-Ringrose-I-Cours}
R. Kadison, J. R. Ringrose,  Fundamentals of the theory of operator
algebras, Vol. I. Elementary theory. Reprint of the 1983 original,
G.S.M., \textbf{15}, American Mathematical Society, Providence, RI,
1997.

\bibitem{Koeber-Schafer}
M. Koeber, U. Sch\"{a}fer. The unique square root of a positive
semidefinite matrix, \textit{Internat. J. Math. Ed. Sci. Tech.},
\textbf{37/8} (2006) 990-992.

\bibitem{Mortad-Oper-TH-BOOK-WSPC}
M. H. Mortad. \textit{An operator theory problem book}, World
Scientific Publishing Co., (2018).


\bibitem{Mortad-square-root-normal}
M. H. Mortad. On the existence of normal square and nth roots of
operators, \textit{J. Anal.},  \textbf{28/3} (2020) 695-703.

\bibitem{Mortad-cex-BOOK}
M. H. Mortad. \textit{Counterexamples in operator theory}, (2022).
Birkh\"{a}user/Springer, Cham.

\bibitem{Mortad-JMAA-p(A)-2023}
M. H. Mortad, Certain properties involving the unbounded operators
$p(T)$, $TT^*$, and $T^*T$; and some applications to powers and
$nth$ roots of unbounded operators, \textit{J. Math. Anal. Appl.},
\textbf{525/2} (2023) Paper No. 127159, 26 pp.

\bibitem{Newcomb}
R. W. Newcomb. On the simultaneous diagonalization of two
semi-definite matrices, \textit{Quarterly of Applied Mathematics},
\textbf{19/2} (1961) 144-146.

\bibitem{Prasolov-Lin-Alg-Book}
V. V. Prasolov. \textit{Problems and theorems in linear algebra}.
Translated from the Russian manuscript by D. A. Le\u{\i}tes.
Translations of Mathematical Monographs, \textbf{134}. American
Mathematical Society, Providence, RI, 1994.

\bibitem{Putnam-square-rt-s-a-logarithm}
C. R. Putnam. On square roots and logarithms of self-adjoint
operators, \textit{Proc. Glasgow Math. Assoc.}, \textbf{4} (1958)
1-2.

\bibitem{Radjavi-Rosenthal-sq-roots-normal}
H. Radjavi, P. Rosenthal.  On roots of normal operators, \textit{J.
Math. Anal. Appl.}, \textbf{34} (1971) 653-664.

\bibitem{Shamir-Cohen-Root-matrices-Optics}
J. Shamir, N. Cohen. Root and power transformations in optics,
\textit{J. Opt. Soc. Amer. A}, \textbf{12/11} (1995) 2415-2423.



\end{thebibliography}
\end{document}